\documentclass[leqno,11pt]{amsart} 
\usepackage{amssymb, amsmath, color, amstext, tikz}
\usepackage[english]{babel}
\usepackage{appendix}
\theoremstyle{plain}
\newtheorem{theo}{\indent\textbf Theorem}[section]
\newtheorem{lemma}[theo]{\indent\textbf Lemma}
\newtheorem{coro}[theo]{\indent\textbf Corollary}
\newtheorem{prop}[theo]{\indent\textbf Proposition}
\theoremstyle{definition} 
\newtheorem{defi}[theo]{\indent\textbf Definition}

\newcommand{\vna}{von Neumann algebra}

\newcommand{\node}{non-degenerate}

\newcommand{\sufa}{subfactor}

\newcommand{\ucp}{unital completely positive}

\newcommand{\syenin}{symmetric enveloping inclusion}

\newcommand{\stin}{standard invariant}

\newcommand{\tII}{type II$_1$}

\newcommand{\TLJ}{Temperley-Lieb-Jones}

\newcommand{\lopa}{loop parameter}
\newcommand{\plal}{planar algebra}
\newcommand{\suplal}{subfactor planar algebra}
\newcommand{\stre}{$*$-representation}
\newcommand{\Hapr}{Haagerup property}
\newcommand{\reHapr}{relative Haagerup property}

\newcommand{\HPmo}{Hilbert $\Pl$-module}
\newcommand{\HTLJmo}{Hilbert $TLJ$-module}
\newcommand{\lowe}{lowest weight}
\newcommand{\Hisp}{Hilbert space}
\newcommand{\GNS}{Gelfand-Naimark-Segal}

\newcommand{\NM}{{N\subset M}}

\newcommand{\TS}{{T\subset S}}

\newcommand{\op}{\text{op}}
\newcommand{\MMop}{M\overline\otimes M^\op}
\newcommand{\MboxM}{M\boxtimes M}

\newcommand{\G}{\mathcal G}
\newcommand{\mH}{\mathcal H}

\newcommand{\Pl}{\mathcal P}
\newcommand{\Bl}{\mathcal B}
\newcommand{\C}{\mathbb C}

\newcommand{\N}{\mathbb N}

\newcommand{\K}{\mathcal K}

\newcommand{\loriar}{\longrightarrow}

\newcommand{\ootimes}{\overline\otimes}

\newcommand{\GrP}{Gr\Pl}

\newcommand{\GrPboxGrP}{\GrP\boxtimes\GrP}

\newcommand{\gO}{\geqslant 0}

\newcommand{\diagxdag}{\begin{tikzpicture}[baseline = .4cm]
\draw (0,-.5)--(0,1.5)--(3,1.5)--(3,-.5)--(0,-.5);\draw (1,0)--(1,1)--(2,1)--(2,0)--(1,0);\draw (0,.5)--(1,.5);\draw (2,.5)--(3,.5);\draw (1.5,1.5)--(1.5,1);\draw (1.5,0)--(1.5,-.5);\node at (1.5,.5) { $x^*$ };\node at (2.1,1.1) { \$ };\node at (.2,1.3) { \$ };
\end{tikzpicture}}
\newcommand{\xstarky}{
x\star_k y = \sum_{a=0}^{ \min(2n,2i) } \sum_{ b=0 }^{ \min(2m,2j) } \begin{tikzpicture}[baseline=.4cm]
\draw (0,-.75)--(5,-.75)--(5,1.75)--(0,1.75)--(0,-.75);\draw (1,0)--(2,0)--(2,1)--(1,1)--(1,0);\draw (3,0)--(4,0)--(4,1)--(3,1)--(3,0);\draw (0,.5)--(1,.5);\draw (2,.5)--(3,.5);\draw (4,.5)--(5,.5);\draw (1.5,-.75)--(1.5,0);\draw (1.5,1)--(1.5,1.75);\draw (3.5,-.75)--(3.5,0);\draw (3.5,1)--(3.5,1.75);\draw (1.75,1) arc (160 : 20 : .75);\draw (1.75,0) arc (-160 : -20 : .75);\node at (1.5,.5) { $ x $ };\node at (3.5,.5) { $ y $ };\node at (2.5 , 1.25) {{ \scriptsize{$a$} }};\node at (2.5 ,-.25 ) {{ \scriptsize{$b$} }};\node at (.5 ,.65 ) {{ \scriptsize{$2k$} }};\node at (2.5 ,.65 ) {{ \scriptsize{$2k$} }};\node at (4.5 ,.65 ) {{ \scriptsize{$2k$} }};
\end{tikzpicture}}

\newcommand{\pixxi}{
\pi_0(x)\xi = \sum_{a,b} \begin{tikzpicture}[baseline=.4cm]
\draw (0,-.75)--(5,-.75)--(5,1.75)--(0,1.75)--(0,-.75);\draw (1,0)--(2,0)--(2,1)--(1,1)--(1,0);\draw (3,0)--(4,0)--(4,1)--(3,1)--(3,0);\draw (1.5,-.75)--(1.5,0);\draw (1.5,1)--(1.5,1.75);\draw (3.5,-.75)--(3.5,0);\draw (3.5,1)--(3.5,1.75);\draw (1.75,1) arc (160 : 20 : .75);\draw (1.75,0) arc (-160 : -20 : .75);\node at (1.5,.5) { $ x $ };\node at (3.5,.5) { $ \xi $ };\node at (2.5 , 1.25) {{ \scriptsize{$a$} }};\node at (2.5 ,-.25 ) {{ \scriptsize{$b$} }};
\end{tikzpicture}}

\newcommand{\xyxi}{\begin{tikzpicture}[baseline = 1.25cm]
\draw (0,0)--(4,0)--(4,3)--(0,3)--(0,0);\draw (.5,.5)--(1.5,.5)--(1.5,1.25)--(.5,1.25)--(.5,.5);\draw(.5,1.75)--(1.5,1.75)--(1.5,2.5)--(.5,2.5)--(.5,1.75);\draw (2.5,1)--(3.5,1)--(3.5,2)--(2.5,2)--(2.5,1);\draw (1,0)--(1,.5);\draw (1,2.5)--(1,3);\node at (1,.8) { $y$ };\node at (1,2.1) { $x$ };\node at (3,1.5) { $\xi$ };
\end{tikzpicture}}
\newcommand{\xixy}{\begin{tikzpicture}[baseline = 1.25cm]
\draw (0,0)--(4,0)--(4,3)--(0,3)--(0,0);\draw (2.5,.5)--(3.5,.5)--(3.5,1.25)--(2.5,1.25)--(2.5,.5);\draw (2.5,1.75)--(3.5,1.75)--(3.5,2.5)--(2.5,2.5)--(2.5,1.75);\draw (.5,1)--(1.5,1)--(1.5,2)--(.5,2)--(.5,1);\draw (3,0)--(3,.5);\draw (3,2.5)--(3,3);\node at (3,.8) { $y$ };\node at (3,2.1) { $x$ };\node at (1,1.5) { $\xi$ };
\end{tikzpicture}}
\newcommand{\diage}{e=\frac{1}{\delta}\ 
\begin{tikzpicture}[baseline=.4cm]
\draw(0,0)--(1.5,0)--(1.5,1)--(0,1)--(0,0);\draw (0,.25) arc (-90:90:.25);\draw (1.5,.75) arc (90:270:.25);
\end{tikzpicture}}

\newcommand{\prodscal}{\langle\ \begin{tikzpicture}[baseline=.4cm]
\draw (.5,-.75)--(5,-.75)--(5,1.75)--(.5,1.75)--(.5,-.75);\draw (1,0)--(2,0)--(2,1)--(1,1)--(1,0);\draw (3.5,0)--(4.5,0)--(4.5,1)--(3.5,1)--(3.5,0);\draw (2.5,0)--(3,0)--(3,1)--(2.5,1)--(2.5,0);\draw (1.75,1) arc (140 : 40 : 1.3);\draw (1.75,0) arc (-140 : -40 : 1.3);\node at (1.5,.5) { $ g_n $ };\node at (4,.5) { $ g_n $ };\node at (2.75,.5) { $\xi^t$ } ;
\end{tikzpicture}\ ,\xi^t \rangle = \langle\ \begin{tikzpicture}[baseline=.4cm]
\draw (.5,-.75)--(4.25,-.75)--(4.25,1.75)--(.5,1.75)--(.5,-.75);\draw (1,0)--(2,0)--(2,1)--(1,1)--(1,0);\draw (2.5,0)--(3,0)--(3,1)--(2.5,1)--(2.5,0);\draw (1.75,1) arc (140 : 40 : 1.3);\draw (1.75,0) arc (-140 : -40 : 1.3);\draw (3.75,1)--(3.75,0);\node at (1.5,.5) { $ g_n $ };\node at (2.75,.5) { $\xi^t$ } ;
\end{tikzpicture}\ ,\xi^t \rangle}

\begin{document}

\title{Hilbert modules over a planar algebra and the Haagerup property}
\maketitle
\begin{center}
{\sc by Arnaud Brothier\footnote{Vanderbilt University, Department of Mathematics, 1326 Stevenson Center Nashville, TN, 37212, USA,\\ arnaud.brothier@vanderbilt.edu} and Vaughan Jones \footnote{Vanderbilt University, Department of Mathematics, 1326 Stevenson Center Nashville, TN, 37212, USA,\\ vaughan.f.jones@vanderbilt.edu}
}
\end{center}

\begin{abstract}\noindent
Given a \suplal\ $\Pl$ and a \HPmo\ of \lowe\ 0 we build a bimodule over the \syenin\ associated to $\Pl$.
As an application we  prove diagrammatically that the \TLJ\ \stin s have the \Hapr. 
This provides a new proof of a result due to Popa and Vaes.
\end{abstract}
\section{Introduction and main results}

Popa initiated the study of approximation properties of subfactors in \cite{Popa_correspondances,Popa_classification_subfactors_amenable,Popa_94_Sym_env_alg}.
To any finite index \sufa  of \tII\ one can associate a combinatorial object called the \stin.
This invariant has been axiomatized as a paragroup, a $\lambda$-lattice, and a \plal\ respectively by Ocneanu, Popa, and the second author \cite{Ocneanu_quant_group_string_galois, popa_system_construction_subfactor, jones_planar_algebra}.
An analogue of quantum doubles for a \sufa\ was introduced by Ocneanu, Longo and Rehren, and Popa( \cite{Ocneanu_quant_group_string_galois, Longo_Rehren_95_Nets_sf, Popa_94_Sym_env_alg}).
The latter construction is called the \syenin.
For the construction of subfactors of Guionnet et al. in \cite{GJS_random_matrices_free_proba_planar_algebra_and_subfactor}, Curran et al. in \cite{Curran_Jones_Shlyakhtenko_14_sym_env_alg}
gave a diagrammatic description of the  \syenin.

Recently, Popa and Vaes introduced a representation theory for \sufa s and \stin s \cite{Popa_Vaes_subfactors}.
They defined the \Hapr\ for a \sufa\ and showed that it depends only on its \stin.
They then showed that the \TLJ\ \stin s have the \Hapr.
(Note, this result was already announced in \cite[Remark 3.5.5]{Popa_betti_numbers_invariants}.)
Their proof uses previous work on discrete quantum groups and the equivalence between the bimodule category associated to the \TLJ\ \stin\ and the representation category of the quantum group PSU$_q(2)$ \cite{DeCommer_Freslon_Yamashita_CCAP}.
Here we give another proof:
\begin{theo}\label{theo:TLJ_Haagerup}
The \TLJ\ \stin\ has the \Hapr\ for any \lopa\ $\delta\in\{2\cos \frac{\pi}{n},\ n\geqslant 3\}\cup [2:\infty)$.
\end{theo}
Our proof only uses  \plal\ technology. 
The idea is that the lowest weight zero annular representations of the planar algebra immediately give "compact" bimodules (which are obvious in the Curran et al. pictures), which tend to the trivial bimodule in the way required by the Popa-Vaes definiton of the Haagerup property.

\subsection*{Acknowledgement}
The first author thanks Dietmar Bisch for many encouragements.
He also thanks Jesse Peterson and Jean-Louis Lhuillier for their patience and guidance.
\section{Preliminaries}\label{sec:preliminaries}

\subsection{The \syenin\ associated to a \suplal}

We refer to \cite{jones_planar_algebra} for more details about \plal s.
We recall the construction of \cite[Section 2]{Curran_Jones_Shlyakhtenko_14_sym_env_alg}.
Note, we define the \syenin\ via the product introduced in \cite[Section 2.1]{Curran_Jones_Shlyakhtenko_14_sym_env_alg} that we call the Bacher product.
Let $\Pl=(\Pl^\pm_n,\ n\gO)$ be a \suplal.
For any $k,n,m\gO,$ let $D_k(n,m)$ be a copy of the vector space $\Pl^+_{n+m+2k}$.
We decorate strings with natural numbers to indicate that they represent a given number of parallel strings.
The distinguished interval of a box is decorated by a dollar sign if it is not at the top left corner.
We will omit unnecessary decorations.
Consider the direct sum
$$Gr_k\Pl\boxtimes Gr_k\Pl:=\bigoplus_{n,m\gO}D_k(n,m)$$
that we equipped with the Bacher product:
$$\xstarky$$
where $x\in D_k(n,m)$ and $y\in D_k(i,j).$
Let $\dag:Gr_k\Pl\boxtimes Gr_k\Pl\loriar Gr_k\Pl\boxtimes Gr_k\Pl$ be the anti-linear involution that sends $D_k(n,m)$ to itself and satisfies
$$x^\dag=\diagxdag\ , \text{ for any } x\in D_k(n,m).$$
 Consider the linear form $\tau:Gr_k\Pl\boxtimes Gr_k\Pl\loriar\C$, which is zero unless $n=m=0$ and sends the unit of $D_k(0,0)$ to $1$.
We have that $(Gr_k\Pl\boxtimes Gr_k\Pl,\star_k,\dag,\tau)$ is an associative $*$-algebra with a faithful tracial state \cite[Corollary 2.3]{Curran_Jones_Shlyakhtenko_14_sym_env_alg}.
Further, $Gr_k\Pl\boxtimes Gr_k\Pl$ acts by bounded operators on the \GNS\ \Hisp\ for $\tau$ \cite[Theorem 2.1]{Curran_Jones_Shlyakhtenko_14_sym_env_alg}.
Let $M_k\boxtimes M_k$ be its \GNS\ completion which is a factor of \tII.
Let $M_k$ be the von Neumann subalgebra of $M_k\boxtimes M_k$ generated by elements of the form
$$\begin{tikzpicture}[baseline = .4cm]
\draw (0,-.5)--(0,1.5)--(3,1.5)--(3,-.5)--(0,-.5);
\draw (1,.25)--(1,1)--(2,1)--(2,.25)--(1,.25);
\draw (0,0)--(3,0);
\draw (0,.5)--(1,.5);
\draw (2,.5)--(3,.5);
\draw (1.5,1.5)--(1.5,1);
\node at (1.5,-.15) {{  \scriptsize{$k$}  }};
\node at (1.5,.5) { $x$ };
\end{tikzpicture}\in D_k(n,0),\ n\gO.$$
Observe, the von Neumann subalgebra of $M_k\boxtimes M_k$ generated by the family of sets $D_k(0,m),\ m\gO$
is isomorphic to $M_k^\op$ and commutes with $M_k$. 
We identify $M_k^\op$ with this von Neumann subalgebra.
Note, we have a unital inclusion of $M_{k-1}$ in $M_k$ by adding two horizontal strings under elements of $M_{k-1}$.
By \cite{GJS_random_matrices_free_proba_planar_algebra_and_subfactor}, $M_{k-1}\subset M_k$ is a \sufa\ of \tII\ with \stin\ isomorphic to the \suplal\ $\Pl$ or its opposite depending on the parity of $k$.
The von Neumann subalgebra of $M_k\boxtimes M_k$ generated by $M_k$ and $M_k^\op$ is isomorphic to $M_k\ootimes M_k^\op$ and the inclusion $$M_k\vee M_k^\op\subset M_k\boxtimes M_k$$
is isomorphic to Popa's \syenin\ associated to the \sufa\ $M_{k-1}\subset M_k$ for any $k\geqslant 1$ \cite{Curran_Jones_Shlyakhtenko_14_sym_env_alg}.
We denote by $M\ootimes M^\op\subset M\boxtimes M$ the inclusion $M_0\vee M_0^\op\subset M_0\boxtimes M_0$ and call it the \syenin\ associated to $\Pl$.
Similarly, we write $D_0(n,m)=D(n,m)$ for any $n,m\gO$.

\subsection{Hilbert modules over a \suplal}

We introduce notations and terminology regarding Hilbert modules over a \suplal.
We refer to \cite{Jones_annular_struct} for more details.
Let us fix a \suplal\ $\Pl$.
An annular tangle $\alpha$ is a tangle in $\Pl$ with the choice of a distinguished internal disc.
We write $Ann\Pl((m,\varepsilon),(n,\epsilon))$ the complex vector space spanned by annular tangles with 2n (resp. 2m) boundary points on its internal (resp. external) disc and where the dollar sign is in a region with shading $\epsilon$ (resp. $\varepsilon$).
A tangle in $Ann\Pl((m,\varepsilon),(n,\epsilon))$ is called a $((m,\varepsilon),(n,\epsilon))$-annular tangle.
Let $A\Pl=(A\Pl((m,\varepsilon),(n,\epsilon)),\ n,m\gO,\ \epsilon,\varepsilon\in\pm)$ be the annular algebroid associated to $\Pl$.
We denote by $\alpha\longmapsto \alpha^\dag$ the anti-linear involution which sends a $((m,\varepsilon),(n,\epsilon))$-annular tangle to a $((n,\epsilon),(m,\varepsilon))$-annular tangle by reflection in a circle half way between the inner and outer boundaries.
A \HPmo\ is a graded vector space $V=(V_n^\pm,\ n\gO)$, where each $V_n^\pm$ is a finite dimensional \Hisp, $A\Pl$ acts on $V$, and the inner product is compatible with this action.
It means that if $\alpha\in A\Pl((m,\varepsilon),(n,\epsilon))$ then it defines a linear map from $V_n^\epsilon$ to $V_m^\varepsilon$ such that $$\langle\alpha(v),w\rangle=\langle v,\alpha^\dag(w)\rangle, \text{ for any } v\in V_n^\epsilon, w\in V_m^\varepsilon.$$
The \lowe\ of a \HPmo\ $V$ is the smallest natural number $n$ such that $V_n^+\neq \{0\}$.

\subsection{\HTLJmo s of \lowe\ 0}\label{sec:TLJ_modules}

Consider the \TLJ\ \plal\ $\Pl=TLJ$ with loop parameter $\delta\geqslant 2$.

Irreducible \HTLJmo s of \lowe\ 0 have been fully classified in \cite{Jones_annular_struct} and in \cite{Graham_Lehrer_TLJ} for the unshaded case.
For any $0<t\leqslant \delta$ there exists a \HTLJmo\ $V(t)=(V(t)^\pm_n,n\gO)$ such that $V(t)_0^+$ is one dimensional and spanned by a unit vector $\xi(t)$ which satisfies 
$$\langle \alpha(\xi(t)),\beta(\xi(t))\rangle=\delta^c t^{2d},$$ 
where $\alpha,\beta$ are annular tangles, $c$ is the number of contractible circles in the $(\pm,\pm)$-annular tangle $\beta^\dag\circ\alpha$ and $d$ is half the number of non-contractible ones.
Those \HTLJmo s will be used to construct \ucp\ maps on the \syenin\ associated to the \TLJ\ \plal.

\section{Hilbert $\Pl$-modules give $(\MMop\subset \MboxM)$-bimodules}\label{sec:modules}

Let $V=(V^\pm_n,\ n\geqslant 0)$ be a \HPmo\ of \lowe\ 0.
For $i,j\gO$, let $\mH_{i,j}$ be a copy of the \Hisp\ $V^+_{i+j}$.
Let $\mH=\bigoplus_{i,j\gO} \mH_{i,j}$ be the \Hisp\ equal to the direct sum of the $\mH_{i,j}$.
In particular, $\mH_{i+1,j-1}$ is orthogonal to $\mH_{i,j}$ in $\mH$.
Consider the dense pre-Hilbert subspace $\K\subset \mH$ spanned by the union of all $\mH_{i,j}$.
We put 
$$
\pixxi\ ,
$$
for any $x\in D(n,m)\subset \GrPboxGrP$ and $\xi\in\mH_{i,j}.$
This defines a representation 
$$\pi_0:\GrPboxGrP\loriar \mathcal L(\K),$$
where $\mathcal L(\K)$ is the algebra of endomorphism of the vector space $\K$.

\begin{prop}\label{prop:bounded_action}
For any $x\in\GrPboxGrP$, $\pi_0(x)$ defines a bounded operator on $\mH$.
Further, the representation $\pi_0$ extends to a normal \stre\ $$\pi:\MboxM\loriar \Bl(\mH).$$
\end{prop}

\begin{proof}
Consider $x$ in $\GrPboxGrP$.
We can prove that $\pi_0(x)$ defines a bounded operator by following a similar argument than \cite[Theorem 3.3]{JSW_orthogonal_approach_planar_algebra}.
We continue to denote by $\pi_0(x)$ its extension to $\mH$.
Let $\xi\in\mH_{0,0}$ be a unit vector and let $\omega_\xi$ be its associated vector state.
Note, $\omega_\xi\circ\pi_0(x)=\tau(x)$ for any $x\in\GrPboxGrP$, where $\tau$ is the unique normal tracial state on $\MboxM$.
Therefore, $\pi_0$ extends to a normal $*$-representation $\pi:\MboxM\loriar B(\mH)$.
\end{proof}

Recall, if $\TS$ is an inclusion of \vna s, then a Hilbert $(\TS)$-module is a couple $(\mH,\xi)$ such that $\mH$ is a Hilbert $S$-module and $\xi$ is a $T$-central vector of $\mH$. 

\begin{coro}
Let $V$ be a \HPmo\ of \lowe\ 0. 
Consider the \Hisp\ $\mH$ constructed above and let $\xi\in \mH_{0,0}$ be a unit vector.
Then, $(\mH,\xi)$ has a structure of Hilbert $(\MMop\subset\MboxM)$-bimodule where the left action is given by $\pi$ and the right action is defined similarly.
\end{coro}

\begin{proof}
Proposition \ref{prop:bounded_action} implies that $\mH$ is a $\MboxM$-bimodule with the action described above.
Consider $x\otimes y^\op\in \GrP\otimes\GrP^\op$, where $GrP\otimes GrP^\op=\MMop\cap \GrPboxGrP$.
Since $\xi\in\mH_{0,0}$, we have
$$(x\otimes y^\op) \cdot \xi= \xyxi \text{ and } \xi \cdot (x\otimes y^\op)= \xixy \ .$$
Those two pictures are isotopic to each other.
Therefore, $(x\otimes y^\op)\cdot \xi=\xi\cdot (x\otimes y^\op).$
By density of $GrP\otimes GrP^\op$ inside $\MMop$, we obtain that $\xi$ is a $\MMop$-central vector.
\end{proof}

\section{The \TLJ\ \stin\ has the \Hapr}\label{sec:Haagerup}

In this article, any inclusion of tracial \vna s will be supposed to be unital and tracial.
We recall the definition of the \reHapr\ due to Boca \cite{Boca_irred_Haagerup}. 
Note, Popa defined a very similar property \cite{Popa_betti_numbers_invariants}.
Those two definitions coincide in the context of Definition \ref{defi:Haagerup}.

\begin{defi}\label{defi:CPAI}
Consider an inclusion of tracial \vna s $\mathcal N\subset (\mathcal M,\tau).$ 
A completely positive approximation of the identity (CPAI) for $\mathcal N\subset (\mathcal M,\tau)$ is a sequence of normal $\mathcal N$-bimodular trace-preserving \ucp\ maps $(\varphi_l:\mathcal M\longrightarrow \mathcal M,\ l\geqslant 0)$ such that $\Vert \varphi_l(x)-x\Vert_2\longrightarrow_l 0,$ for any $x\in\mathcal M$, and the unique continuous extension $\Theta_l\in B(L^2(\mathcal M,\tau))$ of $\varphi_l$ to $L^2(\mathcal M,\tau)$ is in the compact ideal space of $\langle \mathcal M, e_\mathcal N\rangle$.

If such a sequence exists we say that $\mathcal N\subset (\mathcal M, \tau)$ has the \reHapr.
\end{defi}

\begin{defi}\label{defi:Haagerup}\cite{Popa_Vaes_subfactors}
A \sufa\ $\NM$ has the \Hapr\ if its \syenin\ has the \reHapr.
A \stin\ $\G$ has the \Hapr\ if there exists a \sufa\ $\NM$ with \stin\ isomorphic to $\G$ which has the \Hapr.
\end{defi}
Recall that if two \sufa s have isomorphic \stin s, then one of them has the \Hapr\ if and only if the other one has the \Hapr, see \cite[Remark 3.5.5]{Popa_betti_numbers_invariants} or \cite{Popa_Vaes_subfactors}. 

\begin{lemma}
Let $\Pl$ be a \suplal.
Then $\Pl$ has the \Hapr\ if and only if its associated \syenin\ $\MMop\subset\MboxM$ has the \reHapr.
\end{lemma}

\begin{proof}
Consider the \sufa\ $M_0\subset M_1$ defined in Section \ref{sec:preliminaries}.
Its \plal\ is equal to $\Pl$.
Popa's \syenin\ associated to $M_0\subset M_1$ is isomorphic to $M_1\vee M_1^\op \subset M_1\boxtimes M_1$.
Consider the inclusion $M_0\vee M_0^\op\subset M_1\boxtimes M_1$.
Let $e$ be the Jones projection
$\diage\ .$
Note, the compression $e (M_0\vee M_0^\op) e \subset e (M_1\boxtimes M_1) e$ is isomorphic to $\MMop\subset \MboxM$.
Therefore, by \cite[Proposition 2.3 and Proposition 2.4]{Popa_betti_numbers_invariants}, $\MMop\subset \MboxM$ has the \reHapr\ if and only if $M_1\vee M_1^\op \subset M_1\boxtimes M_1$ has the \reHapr.
\end{proof}

\begin{lemma}
Let $TLJ$ be the \TLJ\ \plal\ with a loop parameter $\delta\geqslant 2$ and let $\MMop\subset \MboxM$ be its associated \syenin.
Consider the 2nth-Jones-Wenzl idempotent $g_n\in \text{TLJ}^+_{2n}$ that we identity with its associated element in
$D(n,n)\subset \MboxM.$
Let $L_n\subset L^2(\MboxM)$ be the $\MMop$-bimodule generated by $g_n$.
Then $L_n$ is isomorphic to $X_n\ootimes \overline{X_n}^\op$, where $X_n$ is the irreducible $M_0$-bimodule corresponding the the 2nth vertex in the principal graph of the \sufa\ $M_0\subset M_1$.
Further, $L^2(\MboxM)$ is equal to the direct sum of the bimodule $L_n$.
\end{lemma}

\begin{proof}
We follow an argument in \cite[pp. 120-122]{Curran_Jones_Shlyakhtenko_14_sym_env_alg}.
Let us show that $L_n$ is orthogonal to $L_m$ if $n\neq m$.
This is equivalent to show that for any $x,y\in TLJ$, we have $xg_ny\perp g_m$ in the planar algebra $TLJ$.
But this is obvious.
Observe, the $*$-algebra $\GrPboxGrP$ is generated by the set of Jones-Wenzl idempotents and $GrP\otimes GrP^\op$.
Therefore, $L^2(\MboxM)$ is equal to the direct sum of the bimodules $L_n$.
Consider the $M$-bimodule $X_n\subset L^2(M_n)$ equal to the image of $g_n$ viewed as an element of $TLJ^+_{2n}=M'\cap M_{2n}\subset B(L^2(M_n))$.
We have an isomorphism from $X_n\ootimes \overline{X_n}^\op$ onto $L_n$ given by the tangle which connects the $2n$ side strings of an elements of $X_n$ (resp. $\overline{X_n}^\op$) to the top strings of $g_n$ (resp. the bottom strings of $g_n$).
\end{proof}

\begin{theo}\label{theo:TLJ_Ha}
Let $TLJ$ be the \TLJ\ \plal\ with any loop parameter $\delta\in \{2\cos(\frac{\pi}{n}),\ n\geqslant 3\}\cup [2:\infty)$.
Then $TLJ$ has the \Hapr.
\end{theo}

\begin{proof}
If $\delta=2\cos(\frac{\pi}{n})$ for some $n\geqslant 3$, then $TLJ$ has finite depth. Therefore, its \syenin\ is a \sufa\ of finite index. This implies that $TLJ$ has the \Hapr.
We assume that $\delta\geqslant2$.
We write $T=\MMop$ and $S=\MboxM$.
Consider $0<t<\delta$ and the pointed \HTLJmo\ $(V(t),\xi(t))$ of section \ref{sec:TLJ_modules} where $\xi(t)\in V(t)^+_0$ is a unit vector.
Let $(H^t,\xi^t)$ be its associated ($\TS$)-bimodule as constructed in section \ref{sec:modules}.
Let $Z_t:L^2(S)\loriar H^t$ be the continuous linear map densely defined as follows
$Z_t(x\Omega)=\xi^t\cdot x, \text{ for any } x\in S.$
Define the normal $T$-bimodular \ucp\ map $\phi_t:S\loriar S$ by the formula
$\phi_t(x)=Z_t^*\pi_t(x)Z_t$, where $\pi_t:S\loriar B(H^t)$ is the left action of $S$ on $H^t$.
We will show that the net $(\phi_t,\ 0<t<\delta)$ is the desired approximation of the identity.

Note, the $T$-bimodules $L_n$ are isomorphic to $X_n\ootimes \overline X_n^\op$ for any $n\geqslant 0$.
Hence, they are irreducible and pairwise non-isomorphic.
By Schur's Lemma, there exists a scalar valued function $c_t:\N\loriar\C$ such that $\Theta_t=\sum_{n\geqslant 0} c_t(n) s_n$, where $\Theta_t$ is the unique continuous extension of $\phi_t$ to $L^2(S)$ and $s_n$ is the orthogonal projection from $L^2(S)$ onto $L_n$.
We have the formula 
$$c_t(n)=\frac{\langle\phi_t(g_n),g_n\rangle}{\langle g_n,g_n\rangle}, \text{ for any } n\geqslant 0.$$
Let $\tau_{2n}$ be the non-normalized trace of the C$^*$-algebra $TLJ^+_{2n}$. 
Remark,
$\tau_{2n}(g_n)=\langle g_n,g_n\rangle, \text{ for any } n\geqslant 0.$
Let $q$ be the unique real number bigger than 1 satisfying $q+q^{-1}=\delta$.
It is well known that $\tau_{2n}(g_n)=[2n+1]_q$, where 
$$ [2n+1]_q = \frac{ q^{ 2n+1 } - q^{ -2n-1 } }{ q-q^{ -1 } } $$
is the 2n+1th quantum integer with parameter $q$ \cite[Section 5.1]{Jones_index_for_subfactors}.

We claim that 
\begin{equation}\label{equa:qnumber}
\langle\phi_t(g_n),g_n\rangle=[2n+1]_\omega, \text{ if } n \geqslant 1,
\end{equation}
where $\omega$ is a complex number satisfying $\omega+\omega^{-1}=t$.
Observe, 
\begin{align*}
\langle\phi_t(g_n),g_n\rangle&=\langle g_n\cdot \xi^t,\xi^t\cdot g_n\rangle=\langle g_n\cdot \xi^t \cdot g_n,\xi^t\rangle\\
&= \prodscal.
\end{align*}
Hence, proving the equality (\ref{equa:qnumber}) is a routine computation using the induction formula of \cite[Section 5.1]{Jones_index_for_subfactors} or \cite{Wenzl_projections}.
Therefore, $$c_t(n)=\frac{[2n+1]_\omega}{[2n+1]_q}\text{ for any } n\geqslant 0.$$
Observe,
\begin{align*}
&c_t(n)\loriar 1, \text{ as } t\rightarrow \delta, \text{ for any } n\geqslant 0, \text{ and}\\
&c_t(n)\loriar 0, \text{ as } n\rightarrow \infty, \text{ for any } 0<t<\delta.
\end{align*}
Note, $\tau\circ \phi_t=a_0\tau=\tau.$
Hence, any sequence of real numbers $(0<t_n<\delta,\ n \geqslant0)$ that converges to $\delta$ defines a CPAI $(\phi_{t_n},\ n \geqslant 0)$.
Therefore, $\TS$ has the relative Haagerup property.
\end{proof}

\bibliographystyle{alpha}

\end{document}